\documentclass[12pt]{amsart}
\usepackage{amssymb}
\usepackage[all]{xy}

\newcommand{\mm}{\mathfrak m}
\newcommand{\Z}{\mathbb{Z}}
\newcommand{\R}{\mathbb{R}}
\newcommand{\N}{\mathbb{N}}

\DeclareMathOperator{\chara}{char}

\DeclareMathOperator{\depth}{depth}

\DeclareMathOperator{\GL}{GL}

\DeclareMathOperator{\gin}{gin}

\DeclareMathOperator{\ini}{in}

\DeclareMathOperator{\Rad}{Rad}

\DeclareMathOperator{\reg}{reg}

\DeclareMathOperator{\Tor}{Tor}
\DeclareMathOperator{\htt}{height}

\DeclareMathOperator{\lex}{lex}
\DeclareMathOperator{\rev}{rev}
\DeclareMathOperator{\lin}{lin}
\DeclareMathOperator{\pnt}{\raise 0.5mm \hbox{\large\bf.}}
\DeclareMathOperator{\lpnt}{\hbox{\large\bf.}}

\newcommand{\s}{\; | \;}

\newtheorem{thm}{\bf Theorem}[section]
\newtheorem{lem}[thm]{\bf Lemma}

\newtheorem{prop}[thm]{\bf Proposition}

\newtheorem{quest}[thm]{\bf Question}

\newtheorem{constr}[thm]{\bf Construction}

\theoremstyle{definition}

\newtheorem{rem}[thm]{\bf Remark}
\newtheorem{ex}[thm]{\bf Example}

\let\phi=\varphi

\numberwithin{equation}{section}
\title{Criteria for componentwise linearity}

\author[Uwe Nagel]{Uwe Nagel}
\address{Department of Mathematics, University of Kentucky, 715 Patterson Office Tower, Lexington, KY 40506-0027, USA}
\email{uwe.nagel@uky.edu}

\author{Tim R\"omer}
\address{Universit\"at Osnabr\"uck, Institut f\"ur Mathematik, 49069 Osnabr\"uck, Germany}
\email{troemer@uos.de}

\begin{document}

\thanks{This work was partially supported by a grant from the
Simons Foundation (\#208869 to Uwe Nagel).}

\maketitle

\begin{center}
{\em Dedicated to Winfried Bruns at the occasion of his $65^{th}$ birthday.}
\end{center}

\begin{abstract}
We establish characteristic-free criteria for the componentwise linearity of graded ideals. As applications, we classify the componentwise linear ideals among the Gorenstein ideals, the standard determinantal ideals, and the ideals generated by the submaximal minors of a symmetric matrix.
\end{abstract}


%
%
%
\section{Introduction}

Since their introduction by Herzog and Hibi in \cite{HH-Nagoya},
componentwise linear ideals have been proven  important. They
include all stable monomial ideals. In fact, a first breakthrough
has been the result in \cite{ArHeHi} that, in characteristic zero,
characterizes the componentwise linear ideals as the ideals that
have the same graded Betti numbers as their generic initial ideals
with respect to the reverse lexicographic order. It follows in
particular that Gotzmann ideals are componentwise linear. In fact,
the criterion suggests that componentwise linear ideals may be
viewed as the graded ideals that are analogous to the stable ideals
among the monomial ideals.

It turns out that ideals with extremal homological properties are
often  componentwise linear. For example, the homogeneous ideals of
almost all curves that have maximal Hartshorne-Rao module with
respect to their degree and genus are componentwise linear in
characteristic zero (see Corollary 6.2 in \cite{N-curves}). This
property is also shared by the graded ideals that have maximal
Castelnuovo-Mumford regularity compared to their smallest extended
degree, which is consequence of Theorem 4.4 in \cite{N-compare}.
Furthermore, the tetrahedral curves whose homogeneous ideal is
componentwise linear are characterized in Corollary 4.9 in
\cite{FMN} (see also \cite{FvT} for some further results). The
relevance of componentwise linear ideals for extending classical
results by Zariski is conveyed in \cite{CNR}. However, there are
only few classification results for componentwise linear ideals
available. In this note, we achieve such a classification in three
important classes of ideals: Gorenstein ideals, standard
determinantal ideals, and the ideals generated by the submaximal
minors of a symmetric matrix of height three. Note that our results
are valid over fields of arbitrary characteristic.

After introducing our notation and reviewing some results, in
Section \ref{sec:ideals} we discuss the criteria for componentwise
linearity that are the basis for our subsequent results. The first
criteria are motivated by the above mentioned result of Aramova,
Herzog and Hibi in \cite{ArHeHi},  stating, with the additional
assumption $\chara K=0$, that a graded ideal $I$ is componentwise
linear if and only if $I$ and   $\gin_{<_{\rev}}(I)$ have the same
graded Betti numbers. This remarkable theorem has two drawbacks when
applying it to a specific ideal. Firstly, one would like to omit the
assumption about the characteristic of the field. Secondly, there is
usually no way to ``compute'' the generic initial ideal. Instead one
chooses a map $\phi\in \GL(n;k)$ randomly  enough and compares the
graded Betti numbers of $I$ with the ones of
$\ini_{<_{\rev}}(\phi(I))$. Very likely the ideal
$\ini_{<_{\rev}}(\phi(I))$ is equal to $\gin_{<_{\rev}}(I)$, but
there is no effective way to check this. Especially if one uses
computer algebra systems like Macaulay 2 \cite{MC2}, both problems
occur. We establish solutions to both issues in Theorem
\ref{thm:inarbitrary} and in Theorem \ref{thm:gin}. We also recall a
homological criterion, using the so-called linear part of the
minimal graded free resolution (see Theorem \ref{lem:linearization}).

As a first application, in Section \ref{sec:gorenstein} we describe
all the Gorenstein ideals that are componentwise linear (see Theorem
\ref{thm:comp-lin-Gor}).

In Section \ref{sec:determinantal} we consider standard
determinantal  ideals, i.e.\ the ideals that are generated by the
maximal minors of a homogeneous matrix and have the expected height
(see \cite{BrVe} for a comprehensive treatment). Theorem
\ref{thm:stand-det} characterizes the componentwise linear ideals in
this class. Notice that this includes a characterization of the
Cohen-Macaulay ideals of height two that are componentwise linear.

In our final section we discuss the height three ideals that are
generated by the submaximal minors of a symmetric matrix. The
componentwise linear ideals among these are classified in Theorem
\ref{thm:symm}.

As a consequence of the criterion in Theorem \ref{thm:gin}, every
componentwise linear ideal has the same graded Betti numbers as a
stable monomial ideal. In fact, for each of the componentwise linear
ideals mentioned above we explicitly describe such a stable monomial
ideal. However, we wonder if the existence of such a stable monomial
ideal is sufficient to conclude componentwise linearity. Thus, we
pose the following question.

\begin{quest}
  \label{quest}
Is a graded ideal componentwise linear if and only if it has the
same graded Betti numbers as some stable monomial ideal?
\end{quest}

%
%
%

\section{Criteria}
\label{sec:ideals}

Let $K$ be an arbitrary field, and let $S=K[x_1,\dots,x_n]$ be the
standard graded polynomial ring with maximal graded ideal
$\mm=(x_1,\dots,x_n)$. In this section we recall known results and
present some new methods to check whether a given graded ideal
$I\subset S$ is componentwise linear.

We begin by fixing some notation that we use throughout the paper.
For an $n$-tuple  $a= (a_1,\ldots,a_n)$ of non-negative integers, we
set $x^a = x_1^{a_1} \cdot \ldots \cdot x_n^{a_n}$ and $|a| = a_1 +
\ldots + a_n$. We denote the degree reverse lexicographic order and
the lexicographic order on the monomials in $S$ by $<_{\rev}$ and
$<_{\lex}$, respectively, where $x_1 > x_2 > \ldots > x_n$.  Let $I
\subset S$ be a graded ideal, and let $<$ be any monomial order on
$S$. Then $\ini_<(I)$ and $\gin_< (I)$ denote the initial ideal and
the generic initial ideal of $I$, respectively, with respect to the
order $<$.

If $d$ is any integer and $M$ is a graded $S$-module, then
$M_{\langle d \rangle}$ and $M_{\ge d}$ is the submodule of $M$
generated by the forms of degree $d$ and forms whose degree is at
least $d$, respectively. The Castelnuovo-Mumford regularity of $M$
is denoted by $\reg M$. We use $\beta^S_{i,j} (M)$ and $\beta^S_{i}
(M)$ to denote the graded and total Betti numbers of $M$,
respectively. In particular, $\beta^S_0 (M)$ is the number of
minimal generators of $M$.

The module $M$ is said to have a \emph{$d$-linear resolution} if
$\beta^S_{i,j} (M) = 0$ whenever $j \neq i+d$. Following
\cite{HH-Nagoya}, $M$ is called \emph{componentwise linear} if, for
all integers $d$, the module $M_{\langle d \rangle}$ has a
$d$-linear resolution.

Recall that a monomial ideal $I \subset S$ is said to be
\emph{strongly stable} if, for any monomial $m \in S$, the
conditions $m \in I$ and $x_i$ divides $m$ imply that $x_j \cdot
\frac{m}{x_i}$ is in $I$ whenever $j \le i$. The ideal is called
\emph{stable} if the analogous condition is true whenever $i$ is the
largest index of a variable dividing $m$. We denote this index by
$\max (m)$. Stable ideals admit several combinatorial properties. In
particular, due to Eliahou and Kervaire their graded Betti numbers
are known.

\begin{thm}[\cite{ElKe}]
  \label{thm:Eliahou-Kervaire}
Let $0 \neq I \subset S$ be a stable monomial ideal that is
minimally generated by the monomials $m_1,\ldots,m_t$. Then, for all
$i \ge 1$ and  $j \in \Z$,
\begin{equation}
  \label{eq:ElKe-formula}
\beta^S_{i, i+ j} (S/I) =
\sum_{1 \le k \le t \atop \deg m_k = j} \binom{\max (m_k) - 1}{i-1}.
\end{equation}
\end{thm}

Now we are ready to develop a characteristic-free version of the
main result in \cite{ArHeHi}. In fact, we follow the idea of proofs
in \cite{ArHeHi}, but adapt them to address our need for greater
generality. We need the following variation of \cite[Theorem 1.2
(b)]{ArHeHi}:
\begin{lem}
\label{lem:stablehelper} Let $0\neq I \subset S$ be a graded ideal,
$d\in \N$, and $I=(g_1,\dots,g_m)$ with $\deg g_i =d$ for
$i=1,\dots,m$. Let $<$ be a monomial order on $S$ and assume that
$(\ini_<(g_1),\dots,\ini_<(g_m))$ is a stable monomial ideal. If
$\beta^S_{0,j}(\ini_<(I)) = 0$ for some $j \in \N$ with $d<j$, then
$$
\beta^S_{0,k}(\ini_<(I)) = 0 \text{ for all } k\in \N \text{ with }  k \geq j.
$$
Moreover, if $\beta^S_{0,d+1}(\ini_<(I)) = 0$, then $\ini_<(I)$ and
$I$ have a $d$-linear resolution.
\end{lem}
\begin{proof}
We consider the ideal $I_{\geq j-1}$. Then
$$
I_{\geq j-1}=(x^a\cdot g_i: 1\leq i \leq m,\
a\in \N^n \text{ with } |a|=j-1-d)
$$
is generated in degree $j-1$. We see that
\begin{eqnarray*}
&&(\ini_<(x^a\cdot g_i): 1\leq i \leq m,\ a\in \N^n \text{ with } |a|=j-1-d)\\
&=&
(x^a\cdot \ini_<(g_i): 1\leq i \leq m,\ a\in \N^n \text{ with } |a|=j-1-d)\\
&=&
\mm^{j-1-d}(\ini_<(g_1),\dots,\ini_<(g_m))
\end{eqnarray*}
is again a stable monomial ideal. Since $\ini_<(I_{\geq
j-1})=\ini_<(I)_{\geq j-1}$ it suffices to show that $\ini_<(I_{\geq
j-1})$ has no minimal monomial generator in degrees $\geq j$ in
order to conclude the vanishing statement for $\ini_<(I)$.

So, if we replace $I_{\geq j-1}$ by $I$ we may assume that $I$ is
generated in degree $d$ with $\beta^S_{0,d+1}(\ini_<(I)) = 0$, and
we have to show that $\ini_<(I)$ is generated in degree $d$. Since
$\beta^S_{0,d+1}(\ini_<(I)) = 0$ we know that all $S$-polynomials of
$\{g_1,\dots,g_m\}$ of degree $d+1$ reduce to zero. Using that
$(\ini_<(g_1),\dots,\ini_<(g_1))$ is a stable monomial ideal
generated in degree $d$, it follows from the Eliahou-Kervaire
 formula \eqref{eq:ElKe-formula} for the graded Betti
numbers of that ideal that its first syzygy module is generated in
degree $d+1$. Now we apply \cite[Proposition 2.3.5]{HeHi} to
conclude that $\{g_1,\dots,g_m\}$ is a Gr\"obner-basis of $I$ with
respect to $<$. Thus $\ini_<(I)$ is generated in degree $d$, as
desired.

Moreover, applying once more the Eliahou-Kervaire formula, we see
that $\ini_<(I)$, and thus $I$, both have a $d$-linear resolution.
\end{proof}

Using the latter result we establish the following criterion. Its
proof is similar to the one of one implication in \cite[Theorem
1.1]{ArHeHi},  with appropriate modifications.

\begin{thm}
\label{thm:inarbitrary} Let $0\neq I \subset S$ be a graded ideal,
$<$ a monomial order on $S$ and assume that $\ini_<(I)$ is stable.
If  $\beta^S_{0}(I)=\beta^S_{0}(\ini_<(I))$, then $I$ is
componentwise linear.
\end{thm}
\begin{proof}
Let $d$ be the least degree of a minimal generator of $I$, and let
$D$ be the maximal degree of a minimal generator of $I$. Set
$r=D-d$. We prove the theorem by induction on $r \ge 0$.

Assume that $r=0$. Then $I$ is generated in degree $d$. Since
$\beta^S_{0}(I)=\beta^S_{0}(\ini_<(I))$, we see that $\ini_<(I)$ is
also generated in degree $d$. By assumption, $\ini_<(I)$ is a stable
monomial ideal, and thus it has a $d$-linear resolution by the
Eliahou-Kervaire formula. It follows that also $I$ has a $d$-linear
resolution since $\beta^S_{i,j}(I)\leq \beta^S_{i,j}(\ini_<(I))$ for
all $i,j$.

Next we consider the case $r>0$.  Assume that $I_{\langle d
\rangle}$ has not a $d$-linear resolution. Let $\{g_1,\dots,g_t\}$
be a Gr\"obner basis for $I$ with respect to $<$ where
$\{g_1,\dots,g_m\}$  are all the elements of this Gr\"obner basis of
degree $d$. Then $I_{\langle d \rangle}=(g_1,\dots, g_m)$ and
$\ini_<(I)_{\langle d \rangle}= (\ini_<(g_1),\dots, \ini_<(g_m))$ is
a stable monomial ideal. It follows from Lemma
\ref{lem:stablehelper} applied to $I_{\langle d \rangle}$  that
$\beta_{0,d+1}^S(\ini_<(I_{\langle d \rangle})) \neq 0$. Now
\begin{eqnarray*}
\beta_{0,d+1}^S(\ini_<(I))
&=&
\dim_K \ini_<(I)_{d+1}- \dim_K (\mm \ini_<(I)_{\langle d \rangle})_{d+1}\\
&=&
\dim_K \ini_<(I)_{d+1}- \dim_K (\mm \ini_<(I_{\langle d \rangle}))_{d+1}\\
&>&
\dim_K \ini_<(I)_{d+1}- \dim_K \ini_<(I_{\langle d \rangle})_{d+1}\\
&=&
\dim_K I_{d+1}- \dim_K (I_{\langle d \rangle})_{d+1}\\
&=&
\dim_K I_{d+1}- \dim_K (\mm I_{\langle d \rangle})_{d+1} =\beta_{0,d+1}^S(I).
\end{eqnarray*}
This gives the contradiction $ \beta_{0}^S(\ini_<(I)) = \sum_{j\in
\N}\beta_{0,j}^S(\ini_<(I))
> \sum_{j\in \N}\beta_{0,j}^S(I)
=\beta_{0}^S(I).$ Thus, $I_{\langle d \rangle}$ has a $d$-linear
resolution, as desired. Moreover, we also see  that
$\beta_{0,j}^S(\ini_<(I_{\langle d \rangle}))=0$ for $j>d$, and so
the (stable monomial) ideal $\ini_<(I_{\langle d \rangle})$  has
also a $d$-linear resolution. This implies in particular the
equation $\ini_<(I_{\langle d \rangle}) = \ini_<(I)_{\langle d
\rangle}$ which will be needed below.

Now we claim that $\ini_<(I_{\geq d+1})=\ini_<(I)_{\geq d+1}$ and
$I_{\geq d+1}$ have the same graded $0$-th Betti numbers. Granting
this for the time being and using that $\ini_<(I_{\geq d+1})$ is
again a stable monomial ideal, the induction hypothesis provides
that $I_{\geq d+1}$ is componentwise linear. Thus, $I_{\langle j
\rangle}=(I_{\geq d+1})_{\langle j \rangle}$ has a $j$-linear
resolution for each $j \geq d+1$. This shows that $I$ is
componentwise linear.

It remains to prove the claim about the graded $0$-th Betti numbers.
For $j\leq d$ we have $\beta_{0,j}^S(I_{\geq d+1})=0=\beta_{0,j}^S(\ini_<(I_{\geq d+1}))$. For $j\geq d+1$ we consider the short exact sequence
$$
0 \to I_{\geq d+1} \to I \to K(-d)^{\beta_{0,d}^S(I)} \to 0
$$
and the associated long exact sequence
\begin{eqnarray}
\label{eq:long-seq}
\dots &\to&
\Tor_1^S(I_{\geq d+1},K)_{1+j-1}
\to
\Tor_1^S(I,K)_{1+j-1}
\to
\Tor_1^S(K,K)^{\beta_{0,d}^S(I)}_{1+j-d-1}\\
&\to&
\Tor_0^S(I_{\geq d+1},K)_{j}
\to
\Tor_0^S(I,K)_{j}
\to
\Tor_0^S(K,K)^{\beta_{0,d}^S(I)}_{j-d}
\to 0.  \nonumber
\end{eqnarray}
Analogous sequences exist for  $\ini_<(I_{\geq d+1})=\ini_<(I)_{\geq
d+1}$. If $j>d+1$, then we see that
$$
\beta_{0,j}^S(I_{\geq d+1})
=
\beta_{0,j}^S(I)
=
\beta_{0,j}^S(\ini_<(I))
=
\beta_{0,j}^S(\ini_<(I)_{\geq d+1})
=
\beta_{0,j}^S(\ini_<(I_{\geq d+1}))
$$
because $K$ has a linear resolution as an $S$-module and the second
identity is given by the assumption of the Theorem. Finally, we
consider the case $j=d+1$. Sequence \eqref{eq:long-seq} provides
\begin{eqnarray*}
\beta_{0,d+1}^S(I_{\geq d+1})
&=&
\beta_{0,d+1}^S(I) + n\cdot \beta_{0,d}^S(I)-\beta_{1,1+d}^S(I)\\
&=&
\beta_{0,d+1}^S(\ini_<(I)) + n\cdot \beta_{0,d}^S(\ini_<(I))-\beta_{1,1+d}^S(I_{\langle d \rangle})\\
&=&
\beta_{0,d+1}^S(\ini_<(I)) + n\cdot \beta_{0,d}^S(\ini_<(I))-\beta_{1,1+d}^S(\ini_<(I_{\langle d \rangle}))\\
&=&
\beta_{0,d+1}^S(\ini_<(I)) + n\cdot \beta_{0,d}^S(\ini_<(I))-\beta_{1,1+d}^S(\ini_<(I)_{\langle d \rangle})\\
&=&
\beta_{0,d+1}^S(\ini_<(I)) + n\cdot \beta_{0,d}^S(\ini_<(I))-\beta_{1,1+d}^S(\ini_<(I))\\
&=&
\beta_{0,d+1}^S(\ini_<(I)_{\geq d+1})=\beta_{0,d+1}^S(\ini_<(I_{\geq d+1})).
\end{eqnarray*}
Here the  second identity follows from the assumption of the Theorem
and the fact that $I$ and $I_{\langle d \rangle}$ have the same
graded Betti-numbers $\beta_{i,i+d}$. The third identity is a
consequence of the fact that $\ini_<(I_{\langle d\rangle})$ and
$I_{\langle d \rangle}$ both have a $d$-linear resolution as was
shown above and these ideals have the same Hilbert-function. The
forth identity was also observed above. For the sixth identity we
use analogous arguments applied to $\ini_<(I)_{\geq d+1}$. This
concludes the proof.
\end{proof}

\begin{rem}
\
\begin{enumerate}
\item
As mentioned in the  introduction, to decide whether a given
ideal $I$ is componentwise linear one checks if the graded Betti
numbers of $I$ and $\gin_{<_{\rev}}(I)$ coincide (see, e.g.,
Theorem \ref{thm:gin} below). However, to do this in a concrete
situation one chooses $\phi \in \GL(n;K)$ randomly and compares
the Betti numbers $I$ with the ones of
$\ini_{<_{\rev}}(\phi(I))$ provided that
$\ini_{<_{\rev}}(\phi(I))$ is a stable monomial ideal. The
latter assumption is usually fulfilled if $\chara K=0$ or
$\chara K$ is sufficiently large. Since $I$ is componentwise
linear if and only if $\phi(I)$ is componentwise linear, Theorem
\ref{thm:inarbitrary} shows that this randomized test can really
be used to conclude that $I$ is componentwise linear if all
Betti numbers of $I$ and $\ini_{<_{\rev}}(\phi(I))$ are equal.
\item
The assumption that the initial ideal is stable is essential in
Theorem \ref{thm:inarbitrary}. Usually initial ideals of
componentwise linear ideals do not satisfy this assumption. For
 example, the ideal $I=(x_2,\dots, x_n)$  of $K[x_1,\dots, x_n]$
has a linear resolution given by the Koszul complex on
$x_2,\dots, x_n$. Thus it is componentwise linear. We have
$\ini_<(I)=I$ with respect to any monomial order $<$ on $S$, but
$I$ is not a stable ideal with respect to $x_1>\dots>x_n$.
\item
For an arbitrary monomial order on a polynomial ring there is no
hope that the reverse implication of Theorem
\ref{thm:inarbitrary} is true. Indeed, for an  example consider
$S = \Z/31013\Z[x_1,\dots,x_8]$ and the ideal $I$, generated by
the $2$-minors of the $2\times 4$ matrix
$$
\left[
  \begin{array}{cccc}
    x_1 & x_2 & x_3 & x_4 \\
    x_5 & x_6 & x_7 & x_8
  \end{array}
\right].
$$
It is well-known that $I$ has a 2-linear resolution that is
given by the  Eagon-Northcott complex; see, e.g., \cite[Theorem
2.16]{BrVe}. Thus, $I$ and all $\phi(I)$, for $\phi\in
\GL(8;\Z/31013\Z)$, are componentwise linear ideals. A
computation with Macaulay 2 \cite{MC2} shows, that a randomly
chosen $\phi\in \GL(8;\Z/31013\Z)$ exists such that
$\ini_{<_{\lex}}(\phi(I))$ equals the stable monomial ideal
$$
J=(x_1^2, x_1x_2, x_1x_3, x_1x_4, x_1x_5, x_2^2, x_2x_3^2, x_2x_3x_4, x_2x_3x_5, x_2x_4^3, x_3^4).
$$
This ideal $J$ is the lex-gin of $I$ with very high probability.
But we see that $\beta_0^S(J)=11>6=\beta_0^S(\phi(I))$.
\end{enumerate}
\end{rem}

The reverse implication of Theorem \ref{thm:inarbitrary} is
``generically'' true if we consider the reverse lexicographic order.
The next result is the desired generalization of \cite[Theorem
1.1]{ArHeHi} to any characteristic.

\begin{thm}
\label{thm:gin} Let $0\neq I \subset S$ be a graded ideal. Then the
following statements are equivalent:
\begin{enumerate}
\item
$\gin_{<_{\rev}}(I)$ is stable and
$\beta^S_{i,j}(I)=\beta^S_{i,j}(\gin_{<_{\rev}}(I))$ for all
$i,j$;
\item
$\gin_{<_{\rev}}(I)$ is stable and
$\beta^S_{i}(I)=\beta^S_{i}(\gin_{<_{\rev}}(I))$ for all $i$;
\item
$\gin_{<_{\rev}}(I)$ is stable and
$\beta^S_{0}(I)=\beta^S_{0}(\gin_{<_{\rev}}(I))$;
\item
$I$ is componentwise linear.
\end{enumerate}
\end{thm}
\begin{proof}
Trivially, the implications (i) $\Rightarrow$ (ii) and (ii)
$\Rightarrow$ (iii) are true.

(iii) $\Rightarrow$ (iv): There exists a $\phi \in \GL(n;K)$ such
that $\gin_{<_{\rev}}(I)=\ini_{<_{\rev}}(\phi(I))$. It follows from
Theorem \ref{thm:inarbitrary} that $\phi(I)$ is componentwise
linear. Hence also $I$ is componentwise linear.

(iv) $\Rightarrow$ (i): Let us assume that  $I$ is componentwise
linear. By \cite[Lemma 1.4]{CoHeHi} we know that
$\gin_{<_{\rev}}(I)$ is stable. The proof of the statement about the
Betti numbers is verbatim the same as the (characteristic free)
proof of the corresponding implication of \cite[Theorem
1.1]{ArHeHi}. The only difference is to observe that
$\gin_{<_{\rev}}(I)$ is still componentwise linear because it is a
stable monomial ideal. (In characteristic zero it is even strongly
stable, as used in the proof of \cite[Theorem 1.1]{ArHeHi}.)
\end{proof}

\begin{rem}
\
\begin{enumerate}
\item Since in characteristic zero $\gin_< (I)$ is always strongly
stable, Theorem \ref{thm:gin} recovers indeed the main result of
\cite{ArHeHi}.

\item
In characteristic zero \cite[Theorem 1.1]{ArHeHi} was
generalized by Conca \cite{Co} (see, e.g., Theorem 4.5) in a
different direction.  He gives equivalent vanishing statements
for certain Koszul homology groups. With suitable adjustments,
these statements are also true in positive characteristic.
\item
If one assumes that $\ini_{<_{\rev}}(I)$ and, for all $j$,
$\ini_{<_{\rev}}(I_{\langle j \rangle})$  are stable monomial
ideals, then one can also prove a  result that is analogous to
Theorem \ref{thm:gin}, but where $\ini_{<_{\rev}}(I)$ is used
instead of $\gin_{<_{\rev}}(I)$.
\end{enumerate}
We leave the details of the proofs to the interested reader.
\end{rem}

A monomial ideal $I$ in $K[x_1,\dots,x_n]$ may be considered as a
monomial ideal in   the polynomial ring $F[x_1,\dots,x_n]$, where
$F$ is any other field. So one might ask if the property
``componentwise linear'' for $I$ is independent of the
characteristic of the considered fields. This is not the case as the
following example shows.
\begin{ex}
Let $S=K[x_1,\dots,x_6]$. With the notation from \cite[p.
236]{BH-book}, the Stanley-Reisner ideal of the triangulation
$\Delta$ of the real projective plane $\mathbb{P}^2_{\R}$ is
$$
I_\Delta=
(
x_1x_2x_3,
x_1x_2x_4,
x_1x_3x_5,
x_1x_4x_6,
x_1x_5x_6,
x_2x_3x_6,
x_2x_4x_5,
x_2x_5x_6,
x_3x_4x_5,
x_3x_4x_6)
.
$$
If $\chara K\neq 2$, then $K[\Delta]=S/I_\Delta$ is a
$2$-dimensional Cohen-Macaulay ring. If  $\chara K =2$, then it is
well-known that $K[\Delta]$ is not Cohen-Macaulay. We know that
$\depth K[\Delta]>0$. Thus we get immediately that $\depth
K[\Delta]=1$. Hence the first local cohomology module
$H^1_\mm(K[\Delta])\neq 0$. Since  $\Delta$ is  a triangulation of a
manifold without boundary, it is known that $K[\Delta]$ is
Buchsbaum. Hence $H^1_\mm(K[\Delta])$ has finite length; see, e.g.,
\cite[Beispiel 6.2.3]{Sch} for details. Then $K[\Delta]$ cannot be
sequentially Cohen-Macaulay as such rings can never have a
 nonzero  first local cohomology module of finite length
(see, e.g., \cite[Theorem 2.11]{St}).

We set $[n]=\{1,\dots,n\}$. Let  $\Delta^\ast=\{F\subseteq [n]:
[n]\setminus F \not\in \Delta \}$ be the Alexander dual of $\Delta$,
and let $I_{\Delta^\ast}$ be its Stanley-Reisner ideal. Then, in
this case, $I_{\Delta}= I_{\Delta^\ast}$. Hence,  \cite[Theorem
9]{HeReWe} provides that $I_{\Delta}$ is componentwise linear if
$\chara K \neq 2$ and $I_{\Delta}$ is not componentwise linear if
$\chara K= 2$.
\end{ex}

In the next sections we also need another method to decide whether a
graded ideal is componentwise linear. For this we consider the
following construction.

\begin{constr}\rm
Let $M$ be a finitely generated graded $S$-module. Let $F_{\lpnt}$
be the minimal graded free resolution of $M$ with differential
$\partial$. Since $\partial_n(F_n)\subseteq \mm F_{n-1}$ for all
$n$, we can define subcomplexes $F^i_{\lpnt}$ of $F_{\lpnt}$ by
setting
$$
F^i_{n}=\mm^{i-n} F_n \text{ for all } n\in \Z
$$
together with the differential $\partial$. Here we set $\mm^j=S$ for
$j\leq 0$. Now $F^{i+1}_{\lpnt} \subseteq F^i_{\lpnt}$ for all $i$.
Hence these subcomplexes define a filtration of $F_{\lpnt}$. The
associated graded complex $F^{\lin}_{\lpnt}$ with respect to this
filtration is called the \emph{linear part} $F_{\lpnt}^{lin}$ of
$F_{\lpnt}$.

We may consider $F_{\lpnt}^{lin}$ again as a complex of finitely
generated graded $S$-modules with the following explicit
description. We have $F^{\lin}_{n}=F_n$. The differential of
$F^{\lin}$ can be obtained as follows. Choose homogeneous bases of
the free modules $F_n$ of $F_{\lpnt}$ such that all the maps
$\partial_i$ are described by homogeneous matrices. Then in all
matrices replace each entry by zero whose degree is at least two.
These matrices induce exactly the differential of $F^{\lin}$.

Similarly, for a matrix $A$ whose entries are homogeneous
polynomials of $S$, we define $A^{\lin}$ as the matrix obtained from
$A$ by replacing each entry  by zero whose degree is at least two.
We call $A^{lin}$ the \emph{linearization} of $A$.
\end{constr}

The following result has been shown in \cite[Theorem 3.2.8]{Ro01};
see also \cite[Theorem 5.6]{IyRo} for a shorter proof.

\begin{thm}
  \label{lem:linearization}
Let $M$ be a finitely generated graded $S$-module with minimal
graded free resolution $F_{\lpnt}$. Then the following statements
are equivalent:
\begin{enumerate}
\item
$M$ is componentwise linear;
\item
the complex $F_{\lpnt}^{lin}$ is acyclic.
\end{enumerate}
\end{thm}

%
%
%
\section{Gorenstein ideals}
\label{sec:gorenstein}

As a first application of our criterion for componentwise linearity
using generic initial ideals, we characterize the componentwise
linear Gorenstein ideals in $S=K[x_1,\dots,x_n]$.

\begin{thm}
  \label{thm:comp-lin-Gor}
Let $0 \neq I$ be a graded Gorenstein ideal of $S$. Then $I$ is
componentwise linear if and only if $I$ is a complete intersection
whose minimal generators, except possibly one, are linear forms.
\end{thm}

\begin{proof}
Denote by $c$ the height of $I$.

If $I$ is a complete intersection that is generated by linear forms
and one form of degree $e \geq 1$, then its Castelnuovo-Mumford
regularity is $e$. Using \cite{BaSt}, we conclude that
$\gin_{<_{\rev}}(I)$ also has regularity $e$. Hence Proposition 10
in \cite{ERT} provides that the ideal $(\gin_{<_{\rev}}(I))_{\langle
e \rangle}$ is stable. It follows that $\gin_{<_{\rev}}(I) =
(x_1,\ldots,x_{c-1}, x_c^e)$ because $\gin_{<_{\rev}}(I)$ has only
one minimal generator of degree $e$ if $e > 1$. Since
$\gin_{<_{\rev}}(I)$ is stable, Theorem \ref{thm:gin} yields that
$I$ is componentwise linear.

For the converse, assume that $I$ is componentwise linear. We will
use the fact that the minimal graded free resolution of $S/I$ over $S$ is
self-dual, i.e., dualizing with respect to $S$ the deleted
resolution of $S/I$ gives again the resolution of $S/I$, up to a
degree shift. Suppose that $S/I$ does not have a $(c-1)$-syzygy of
degree $c-1$, that is, $\Tor_{c-1}^S (S/I, K)_{c-1} = 0$. By
self-duality of the resolution, this means that $I$ does not have a
minimal generator of degree $e+1$, where $e = \reg (S/I)$ is the
Castelnuovo-Mumford regularity of $S/I$. However, since $I$ is
Gorenstein, $\dim_K \Tor_c^S (S/I, K)_{c+e} = 1$. Using the
Eliahou-Kervaire resolution (see \cite{ElKe} and above), it follows
that $I$ does not have the graded Betti numbers of a stable ideal.
Hence, by Theorem \ref{thm:gin}, $I$ cannot be componentwise linear.
This contradiction shows that $I$ must have a $(c-2)$-syzygy of
degree $c-1$. This syzygy is determined by the linear forms in $I$.
The ideal generated by these linear form is resolved by the Koszul
complex. It follows that $I$ must contain a regular sequence
consisting of $c-1$ linear forms. Since $I$ has height $c$, $I$ must
be a  complete intersection as claimed.
\end{proof}

\begin{rem}
  \label{rem:Gor}
\
\begin{enumerate}
  \item The above proof shows that the strongly stable ideals that
have the same graded Betti numbers as a Gorenstein ideal are
precisely the monomial ideals of the form
\[
(x_1,\ldots,x_{c-1}, x_c^e),
\]
where $e$ is any positive integer. Moreover, these ideals are
all the generic initial ideals of componentwise linear
Gorenstein ideals. In particular, these generic initial ideals
are always strongly stable, regardless of the characteristic of
the ground field.

  \item Notice that the result of our characterization of the
  Gorenstein componentwise linear ideals is similar to Goto's
one for the  characterization of integrally closed  complete
intersections (see \cite{Goto}).
\end{enumerate}
\end{rem}

%
%
%
\section{Standard determinantal ideals}
\label{sec:determinantal}

As a second class of examples, we treat standard determinantal
ideals now. Recall that a {\em standard determinantal ideal} in $S$
is an ideal generated by the maximal minors of a $(m+c-1) \times m$
homogenous matrix that has (the expected) height $c \ge 1$. A
homogeneous matrix is a matrix describing a graded homomorphism of
free graded $S$-modules.

Let now $A = (a_{i j})$ be a homogeneous $(m+c-1) \times m$ matrix.
Its {\em degree matrix} is $\partial A = (\deg a_{i j})$. The
homogeneity of $A$ means that its entries $a_{i j} \in S$ are
homogeneous polynomials that satisfy the conditions
\begin{equation*}
  \deg a_{i, i} + \deg a_{i+1, i+1} =
\deg a_{i+1, i} + \deg a_{i, i+1}.
\end{equation*}
Permuting rows and columns of $A$ we may assume that the entries of
$\partial A$ do not increase from left to right and from bottom to
top, that is,
\begin{equation}
  \label{eq:deg-order}
\deg a_{i, j} \geq \deg a_{i, j+1} \quad \text{and} \quad \deg a_{i, j}
\geq \deg a_{i-1, j}.
\end{equation}

Throughout this section we assume that $\partial A$  satisfies these
conditions. Observe that we also may assume that the entries of $A$
are all in the irrelevant maximal ideal $\mm = (x_1,\ldots,x_n)$.

The following is the main result of this section.

\begin{thm}
  \label{thm:stand-det}
Let $I = I_m (A)$ be the standard determinantal ideal of height $c$
that is generated by the maximal minors of the homogeneous matrix
$A$. Assume that its degree matrix $\partial A$ satisfies
\eqref{eq:deg-order}. Then $I$ is componentwise linear if and only
if
\begin{itemize}
\item[(i)] $c = 1$; \; or

\item[(ii)] $c = 2$ and the entries on the main diagonal of $A$ all
have degree one, i.e., for all $i$,  $\deg a_{i, i} = 1$; \; or

\item[(iii)] $c \geq 3$ and the entries in all rows of $A$, except
possibly the last one, have degree one.
\end{itemize}
\end{thm}

We will divide the proof of Theorem \ref{thm:stand-det} into several
steps. We  begin by deriving  necessary conditions.

\begin{lem}
  \label{lem:main-diag}
Let $A$ be an $(m+1) \times m$ homogeneous matrix with entries in
$\mm$. Assume that the degree of one of the entries on the main
diagonal of $A$ is at least two. Then all the maximal minors of its
linearization $A^{\lin}$ vanish.
\end{lem}

\begin{proof}
Let $s \geq 0$ be the least integer such that $\deg a_{s+1, s+1}
\geq 2$.  Then the linearization of $A$ is of the form
\[
A^{\lin} = \begin{bmatrix}
  B_1 & B_2 \\
  0   & B_3
\end{bmatrix},
\]
where $B_1$ is an $s \times s$ matrix and all entries in the first
column  of $B_3$ are zero. Consider now the square matrix
$\widetilde{A^{lin}}$ that is obtained from $A^{\lin}$ by deleting
$c-1$ of its last $c-1$ rows:
\[
\widetilde{A^{\lin}} = \begin{bmatrix}
  B_1 & B_2 \\
  0   & \widetilde{B_3}
\end{bmatrix}.
\]
Then $\widetilde{B_3}$ is a square matrix whose first column entries
are  all zero. Thus, we get \[ \det \widetilde{A^{\lin}} = \det B_1
\cdot \det \widetilde{B_3} = 0.
\]
Next, we consider a matrix $\overline{A^{\lin}}$ obtained from
$A^{\lin}$ by deleting $c-1$ rows  such that $t$ of these rows ($1
\leq t \leq s$) are among the first $s$ rows of $A^{\lin}$. It is of
the form
\[
\overline{A^{\lin}} = \begin{bmatrix}
  \overline{B_1} & \overline{B_2} \\
  0   & \overline{B_3}
\end{bmatrix},
\]
where $\overline{B_1}$ is an $s\times s$ matrix whose last row has
only  zero entries. It follows that
\[
\det \overline{A^{\lin}} = \det \overline{B_1} \cdot \det \overline{B_3}
 = 0.
\]
Thus, we have shown that all $m$-minors of $A^{\lin}$ vanish.
\end{proof}

We can push these arguments further when we have a weaker
assumption.

\begin{lem}
  \label{lem:side-diag}
Let $A$ be an $(m+1) \times m$ homogeneous matrix with entries in
$\mm$. Assume that one  of the entries on the diagonal below  the
main diagonal of $A$ is at least two. Then the maximal minors of its
linearization $A^{\lin}$ generate an ideal whose height is at most
one.
\end{lem}

\begin{proof}
Let $s \ge 1$ be the least integer such that $\deg a_{s+1, s} \geq
2$. Then the linearization of $A$ is of the form
\[
A^{\lin} = \begin{bmatrix}
  B_1  &  B_2 \\
  0   & B_3
\end{bmatrix},
\]
where $B_1$ is an $s \times s$ matrix and $B_3$ is an $(m+c-1-s)
\times (m-s)$ matrix. Considering the $m$-minors of $A^{\lin}$,  we
distinguish two cases depending on the $c-1$ rows that are deleted
from $A^{\lin}$ to obtain a square matrix. First assume that one of
the deleted rows has an index which is at most $s$. Then, we get as
in the proof of Lemma \ref{lem:main-diag} that the corresponding
$m$-minor vanishes. Second, assume that only rows whose index is at
least $s+1$ have been deleted. Then $\det B_1$ divides the resulting
$m$-minor.

Thus, we have shown that all $m$-minors of $A^{\lin}$ are divisible
by $\det B_1$, and our claim follows.
\end{proof}

We also need the following observation.

\begin{lem}
  \label{lem:block-factor}
Let $A$ be an $(m+1) \times m$ homogeneous matrix with entries in
$\mm$ such that $\htt I_m(A) = 2$. Assume that each entry on the
main diagonal of $A$ has degree one and that the degree of one entry
on the  diagonal below the main diagonal of $A$ is at least two.
Then $A^{\lin}$ has a non-vanishing  maximal minor.
\end{lem}

\begin{proof}
Denote by $\overline{A}$ the square matrix obtained from $A$ by deleting its last row.
Let $a_{i+1, i}$ be an entry whose degree is at least two. If $i < m$, then we get by homogeneity of $A$ that
\[
2 = \deg a_{i, i} + \deg a_{i+1, i+1} =
\deg a_{i+1, i} + \deg a_{i, i+1}.
\]
We conclude that $\deg a_{i, i+1} \leq 0$. Thus, $a_{i, i+1} = 0$.
Hence, using also Condition \eqref{eq:deg-order}, the matrix
$\overline{A}$ has the following shape
\[
\overline{A} =\begin{bmatrix}
A_1 & & 0 \\
    & \ddots \\
*   & & A_t
\end{bmatrix},
\]
where each matrix $A_j$ is a square matrix whose entries on the
diagonal below the main diagonal all have degree one. As $I_m (A)$
has the expected height, each of the maximal minors of $A$ is
non-trivial by the Hilbert-Burch theorem. It follows that
\[
0 \neq \det \overline{A}= \det A_1 \cdot \ldots  \cdot \det A_t.
\]
Since the entries on the main diagonal of $A_j$ and the diagonal
below have degree one, all entries of $A_j$ have degree one, and
thus $A_j^{\lin} = A_j$. This provides,
\[
\det \overline{A^{\lin}} = \det A_1 \cdot \ldots  \cdot \det A_t,
\]
which we have seen to be non-trivial.
\end{proof}

Our proof of Theorem \ref{thm:stand-det} is based on the following
result.

\begin{prop}
  \label{lem:crit-stand-det}
Let $I = I_m (A)$ be a standard determinantal ideal  of height $c$,
where $A$ is an $(m+c-1) \times m$ homogeneous matrix with entries
in $\mm$. Then $I$ is componentwise linear if and only if the height
of $I_m (A^{\lin})$ is at least $c-1$.
\end{prop}

\begin{proof}
A minimal graded free resolution of $I$ is given by the
Eagon-Northcott complex  ${F}_{\lpnt}$ (see, e.g., \cite[Theorem
2.16]{BrVe}). Denote by $B_1,\ldots,B_{c-1}$ the matrices that
describe the maps in the deleted resolution of $S/I$. For $i \in
\{1,\ldots,c-1\}$, consider the matrix $B_i$. If $r_i$ is the rank
of $B_i$, then $\Rad I_{r_i} (B_i) = \Rad I$, where $I_{r_i} (B_i)$
is the ideal generated by the $r_i$-minors of $B_i$.  Furthermore,
the construction of the Eagon-Northcott complex (see, e.g.,
\cite{BrVe} or \cite{E}) provides that each entry of $B_i$ is either
zero or equal to an entry of $A$. It also follows that the maps in
the deleted Eagon-Northcott complex associated to the matrix
$A^{\lin}$ are given by the matrices
$B_1^{\lin},\ldots,B_{c-1}^{\lin}$. Thus, Theorem A2.10.b in
\cite{E} provides, for all $i$,
\[
\Rad I_{r_i} (B_i^{\lin}) = \Rad I_m (A^{\lin}).
\]
Using the Buchsbaum-Eisenbud acyclicity criterion, we see that
${F}_{\lpnt}^{\lin}$ is acyclic if and only if $\htt I_m (A^{\lin})
\geq c-1$. Now the claim follows by Theorem \ref{lem:linearization}.
\end{proof}

It is indeed possible that the height of $I_m (A^{\lin})$ is $c-1$
and $I_m (A)$ is componentwise linear of height $c$.

\begin{ex}
  \label{exa:codim-drops}
Consider the ideal $I = (x^2, x y, y^3)$ in $K[x, y]$. It is
strongly stable, and thus it is componentwise linear. Notice that $I
= I_2 (A)$ is standard determinantal of height two, where
\[
A = \begin{bmatrix}
  y & 0 \\
  -x & y^2 \\
  0 & - x
\end{bmatrix}.
\]
Then
\[
A^{\lin} = \begin{bmatrix}
  y & 0 \\
  -x & 0 \\
  0 & - x
\end{bmatrix},
\]
and $I_2 (A^{\lin}) = (x^2, x y)$ has height one.
\end{ex}

We are ready to establish the main result of this section.

\begin{proof}[Proof of Theorem \ref{thm:stand-det}]
Note that every unmixed ideal of height one is principal, thus
componentwise linear. Thus, let now $c \ge 2$.

Recall that we may assume that all entries of $A$ are in $\mm$ and
that we have ordered the entries of $A$ according to their degrees.
By assumption, $I$ has the expected height, which then implies that
all entries on the main diagonal of $A$ have positive degrees.
Otherwise, a calculation shows that the minor using the first $m$
rows of $A$ is zero; a contradiction. We now distinguish two cases.

First assume that the height $c$ is two. If $I$ is componentwise
linear, then, by Proposition \ref{lem:crit-stand-det}, the height of
the  ideal $I_m (A^{\lin})$ is at least one. Hence, Lemma
\ref{lem:main-diag} shows that the degrees of the entries on the
main diagonal of $A$ are at most one, so they must be equal to one,
as claimed.

For the converse, assume that this degree condition is  satisfied.
If all entries of $A$ have degree one, then $I$ has a linear free
resolution, thus it is componentwise linear. If all the entries on
the diagonal below the main diagonal of $A$ have degree one, then
the homogeneity of $A$ implies that all its entries have degree one.
Thus, we are left to consider the case, where one of the entries on
the diagonal below the main diagonal of $A$ has a degree which is at
least two. Hence, Lemma \ref{lem:block-factor} applies and shows
$I_m (A^{\lin}) \neq 0$. Since $S$ is an integral domain,
Proposition \ref{lem:crit-stand-det} implies that again $I$ must be
componentwise linear. This concludes the argument if $c=2$.

Second, assume that $c \geq 3$. If $I$ is componentwise linear,
then,  by Proposition \ref{lem:crit-stand-det}, $\htt I_m(A^{\lin})
\geq c-1 \geq 2$. Thus, Lemmas \ref{lem:main-diag} and
\ref{lem:side-diag} provide that the entries on the main diagonal of
$A$ and the diagonal below it all have degree one. Now, the
homogeneity of $A$ implies that each entry in the first $m+1$ rows
of $A$ has degree one and that the entries in any other row of $A$
have the same degree, depending on the row. If more than one row of
$A$ contains entries whose degrees are at least two, then the two
last rows of $A^{\lin}$ consists of zero. It follows that $I_m
(A^{\lin})$ is the ideal generated by the $m$-minors of an $(m+c-3)
\times m$ matrix. Thus, $\htt I_m (A^{\lin}) \leq c-2$, which is a
contradiction by Proposition \ref{lem:crit-stand-det} and
establishes the necessity of the claimed degree conditions.

Finally, assume that $A$ satisfies these degree conditions. This
means  that $B^{\lin} = B$, where $B$ is the matrix obtained from
$A$ by deleting its last row. Bruns's generalized principal ideal
theorem \cite{Bruns} provides $\htt I_m (B) = c-1$ (see also Remark
2.2 in \cite{BrVe}). Therefore, we get
\[
\htt I_m (A^{\lin}) \geq \htt I_m (B^{\lin}) = \htt I_m (B) = c-1.
\]
Hence $I$ is componentwise linear by  Proposition
\ref{lem:crit-stand-det}, and the argument is complete.
\end{proof}

\begin{rem}
  \label{rem:two-strands}
Theorem \ref{thm:stand-det} provides in particular that the minimal
graded free resolution of a componentwise linear standard
determinantal ideal whose height is at least three has at most two
strands.
\end{rem}

With respect to Theorem  \ref{thm:gin},  it is interesting to find a
strongly stable ideal that has the same graded Betti numbers as a
given componentwise linear standard determinantal ideal. If the
height is two, the answer is clear because each strongly stable
Cohen-Macaulay ideal of height two is standard determinantal. If the
height is larger, then there are restrictions.

\begin{prop}
  \label{prop:stable-ideal-for-stand-det}
Adopt the assumptions of Theorem \ref{thm:stand-det}. Set $e := \deg
a_{m,m+c}$. If $c \geq 3$ and $I$ is componentwise linear, then $I$
has the same graded Betti numbers as the strongly stable ideal
\[
J = (x_1,\ldots,x_c)^{m-1} \cdot (x_1,\ldots,x_{c-1}, x_c^e).
\]
\end{prop}

\begin{proof}
It is easily checked that the ideal $J$ is strongly stable. Thus,
the claim follows by comparing the graded Betti numbers of the
Eagon-Northcott resolution of $I$ (see, e.g., \cite[Theorem
2.16]{BrVe}) with the ones of the Eliahou-Kervaire resolution of
$J$.
\end{proof}

%
%

\section{Minors of a symmetric matrix}
\label{sec:symmetric}

In this section we combine the approaches in the two previous
sections in order to determine
 the componentwise linear ideals of height three that are
generated by the submaximal minors of a symmetric matrix.

Throughout this section $A$ will denote a homogeneous symmetric $m
\times m$ matrix. Its degree matrix can be written as
\[
\partial A =
\begin{bmatrix}
2 d_1 & d_1 + d_2 & \ldots & d_1 + d_m \\
d_1 + d_2 & 2 d_2 & \ldots & d_2 + d_m \\
\vdots & \vdots &    \ddots      & \vdots \\
d_1 + d_m & d_2 + d_m & \ldots & 2 d_m
\end{bmatrix}
\]
for some half integers $d_1,\ldots,d_m$. We may and will always
assume that
\[
d_1 \leq d_2 \leq \cdots \leq d_m.
\]
The following is the main result of this section.

\begin{thm}
  \label{thm:symm}
Let $A$ be an $m \times m$ ($m \ge 2$) homogeneous symmetric matrix
with entries in $\mm$ such that $I := I_{m-1} (A)$ has height three.
Then $I$ is componentwise linear if and only if either all entries
of $A$ have degree one or $m$ is odd and  the degree matrix $B$ of
$A$ is of the form
\[
\begin{bmatrix}
B_1 & B_2 \\[2pt]
B_2^T & B_3
\end{bmatrix},
\]
where $B_2$ is an $\frac{m-1}{2} \times \frac{m+1}{2}$ matrix whose
entries are all  one, the entries of $B_1$ are all non-positive, and
the entries of $B_3$ are all at least two.
\end{thm}

As preparation for the proof we derive some estimates for the height
of certain ideals. Define
\[
s := \max \{j \s d_j = d_1\}
\]
and
\[
t := \min \{j \s d_j = d_m\}.
\]
Note that $s < t$ if $d_1 < d_m$. We begin now with an upper bound.

\begin{lem}
  \label{lem:symm-necessary}
Assume that $A$ satisfies $d_1 \leq 0,\; d_m \geq 2,\; d_1 + d_m =
1$, and $s+t \in \{m, m+1\}$. If $t \geq s+2$, then the height of
$I_{m-1} (A^{\lin})$ is at most one.
\end{lem}

\begin{proof}
The assumptions imply that the linearization of $A$ has the form
\[
A^{\lin} =
\begin{bmatrix}
0 & 0   & A_1 \\[2pt]
0 & A_2 & 0   \\[2pt]
A_1^T & 0 & 0
\end{bmatrix},
\]
where $A_1$ is an $s \times (m-t+1)$ matrix and $A_2$ is a symmetric
$(t-s-1) \times (t-s-1)$ matrix. We distinguish two cases.
\smallskip

\noindent {\sf Case 1.} Assume that $m-t+1 = s$, i.e., $A_1$ is a
square matrix. Then we claim that each generator of $I_{m-1}
(A^{\lin})$ is divisible by $\det A_1$.

To this end we consider the maximal minors of $A^{\lin}$. Let $B$ be
the matrix obtained from $A^{\lin}$ by deleting row $i$ and column
$j$. By symmetry, it suffices to consider the case $i \leq j$. If $i
\leq s$ and $t \leq j$, then the block structure of $B$ provides
$\det B = 0$.  Otherwise, it follows that $\det A_1$ is a factor of
$\det B$ (where we include the possibility $\det B = 0$). This
establishes the claim of Case 1.
\smallskip

\noindent {\sf Case 2.} Assume that $A_1$ is not a square matrix,
that is, $m-t = s$. Then we claim that each generator of $I_{m-1}
(A^{\lin})$ is divisible by $\det A_2$.

Again this follows by carefully considering the maximal minors of
$A^{\lin}$. If $B$ is the square matrix obtained from $A^{\lin}$ by
deleting row $i$ and row $j$ such that at least one of $i$ and $j$
is in $\{s+1,\ldots,t\}$, then $\det B = 0$. In all other cases
$\det A_2$ is a factor of $\det A^{\lin}$. Our claim follows.
\smallskip

Thus, we have shown that in either case $\htt I_m (A^{\lin}) \leq
1$.
\end{proof}

We also need a lower bound.

\begin{lem}
  \label{lem:symm-sufficient}
Assume that $A$ is a symmetric $(2s+1) \times (2s+1)$ matrix of the
form
\[
A =
\begin{bmatrix}
0 & A_1 \\[2pt]
A_1^T & A_2
\end{bmatrix},
\]
where $A_1$ is an $s \times (s+1)$ matrix. If the ideal generated by
the maximal minors of $A$ has (expected) height three then the ideal
generated by the maximal minors of $A_1$ has (expected) height two.
\end{lem}

\begin{proof}
Let $B$ be the matrix obtained from $A$ by deleting its last row.
Then, using our assumption on $I_{2s} (A)$,  \cite[Theorem 1.22]{Go}
provides
\[
\htt I_{2s} (B) \geq \htt I_{2s} (A) - 1 \geq 2.
\]
The block structure of $A$ implies that each $2s$-minor of $B$ is a
linear combination of $s$-minors of $A_1$. It follows that
\[
\htt I_s (A_1) \geq \htt I_{2s} (B) \geq 2.
\]
Hence we must have equality because the  height of $I_s (A_1)$ is at
most two.
\end{proof}

As further preparation, we establish a result that is analogous to
Proposition \ref{lem:crit-stand-det}.

\begin{lem}
  \label{lem:crit-symm}
Assume that the height of $I = I_{m-1} (A)$ is three.   Then $I$ is
componentwise linear if and only if the height of $I_{m-1}
(A^{\lin})$ is at least two.
\end{lem}

\begin{proof}
To each symmetric $m \times m$  matrix $A$, J\'ozefiak \cite{Jo}
provides a complex ${F}_{A}$ of length three that is a minimal
graded free resolution of  $I_{m-1} (A)$ if it has height three. The
construction of this complex commutes with passing to its linear
part, that is, $({F}_{A})^{\lin} = {F}_{A^{\lin}}$. Using Lemma
\ref{lem:linearization}, our claim now follows as in the proof of
Proposition \ref{lem:crit-stand-det}.
\end{proof}

We are ready to establish the main result of this section.

\begin{proof}[Proof of Theorem \ref{thm:symm}]
We begin with showing the necessity of the stated degree conditions.
So assume $I$ is componentwise linear. By \cite[Theorem 3.1]{Jo},
$I$ has a minimal graded free resolution of the form
\begin{eqnarray}
   \label{eq:res-symm}
&&\lefteqn{ 0 \to \bigoplus_{1 \leq i < j \leq m} S(-d_i - d_j - D) \to
(\bigoplus_{1 \leq i, j \leq m} S(-D - d_i + d_j))/S (-D)  }\\
 & &  \hspace*{6.5cm} \to
 \bigoplus_{1 \leq i \leq j \leq m} S (d_i + d_j - D) \to S \to S/I \to 0,
 \nonumber
\end{eqnarray}
where $D := 2 (d_1 + \ldots + d_m)$. Denote by $M_i$ the maximal
degree of a generator of the $i$th free module in this resolution.
Then we get
\[
M_1 = D - 2 d_1, \; M_2 = D + d_m - d_1,\; \text{and}\;
M_3 = D + d_m + d_{m-1}.
\]
Since $I$ is componentwise linear, Theorem \ref{thm:gin} together
with the Eliahou-Kervaire resolution \cite{ElKe} provides
\[
M_3 = M_2 + 1 = M_1 + 2.
\]
It follows that we must have
\[
1 = d_1 + d_m = d_1 + d_{m-1}.
\]
Thus,  in particular $t \leq m-1$. If $d_1 = \cdots = d_m =
\frac{1}{2}$, then we are done. Otherwise, we have
\[
d_1 = \cdots = d_s \leq 0
\]
and
\[
d_t = \cdots = d_m \geq 1.
\]
In particular, the above free resolution has at least three strands.
Now denote by $n_i$ the number of generators of degree $M_i$ of the
$i$th free module in the Resolution \eqref{eq:res-symm}. Then we get
\[
n_1 = \binom{s+1}{2},\; n_2 = s (m-t+1),\; \text{and}\;
n_3 = \binom{m-t+1}{2}.
\]
Using again that $I$ must have the same graded Betti numbers as a
 stable ideal, we get
\begin{equation}
  \label{eq:syzygy-relation}
\quad n_1 + n_3 =  n_2
\end{equation}
because in a strongly stable ideal the pure power of $x_1$ occurs
among the minimal generators of least degree, but the resolution of
$I$ has more than one strand.  Condition \eqref{eq:syzygy-relation}
is equivalent to
\[
s \cdot (m-t-s) = (m-t+1) \cdot (m-t-s).
\]
Hence, we have either $m = s+t$ or $m = s+t-1$. Therefore, Lemma
\ref{lem:symm-necessary} applies if $t \geq s+2$. Together with
Lemma \ref{lem:crit-symm} it shows that then $I$ cannot be
componentwise linear. We conclude that $t = s+1$. Using that $s+t
\in \{m, m+1\}$, we get
\[
s= \frac{m-1}{2}\; \text{and}\; t = \frac{m+1}{2} \quad
\text{if $m$ is odd}
\]
and
\[
s= \frac{m}{2} \; \text{and}\; t = \frac{m}{2} + 1 \quad
\text{if $m$ is even.}
\]
However, if $m$ is even, we see that the linearization of $A$ is of
the form
\[
A^{\lin} =
\begin{bmatrix}
0 & A_2 \\[2pt]
A_2^T & 0
\end{bmatrix},
\]
where $A_2$ is an $s \times s$ matrix. It follows that all
submaximal minors of $A^{\lin}$ are multiples of $\det A_2$, which
is a contradiction by Lemma \ref{lem:crit-symm}. Hence, we have
shown that $m$ is odd, $t = s+1 = \frac{m+1}{2}$, $d_1 \leq 0$, and
$d_m \geq 2$, as required.

To show the converse, assume that these degree conditions are
satisfied. If all entries of $A$ have degree one, then the minimal
graded free resolution of $I$ is linear, hence $I$ is componentwise linear.
Thus, it remains to consider the case where $d_1 \leq 0$. Since by
assumption all entries of $A$ are in $\mm$, entries of non-positive
degree must be zero. Thus, $A$ is of the form
\[
A =
\begin{bmatrix}
0 & A_2 \\[2pt]
A_2^T & A_3
\end{bmatrix},
\]
where $A_2$ is an $\frac{m-1}{2} \times \frac{m+1}{2}$ matrix whose
entries all have degree one and where the degree of each entry of
$A_3$ is at least two. Hence its linearization is
\[
A^{\lin} =
\begin{bmatrix}
0 & A_2 \\[2pt]
A_2^T & 0
\end{bmatrix}.
\]
Denote by $J$ the ideal that is generated by the maximal minors of
$A_2$. Then the above block structure implies
\[
 I_{m-1}(A^{\lin}) = J^2.
\]
By Lemma \ref{lem:symm-sufficient}, the ideal $J$ has height two. It
follows that $\htt I_{m-1}(A^{\lin}) = 2$. Now we conclude by Lemma
\ref{lem:crit-symm} that $I$ is componentwise linear, which finishes
the proof.
\end{proof}

\begin{rem}
  \label{rem:strands-symm}
Theorem \ref{thm:symm} provides that if $I$ is componentwise linear,
then its minimal graded free resolution has either one or three
strands.
\end{rem}

As in the previous section we conclude with determining a strongly
stable ideal with the same graded Betti numbers as the given
componentwise linear ideal. As a preparation we need the following
observation.

\begin{lem}
  \label{lem:stable-ideal-for-symm}
Given any positive integer $s$, there are unique integers $t$ and
$r$ such that $t \ge -1,\; 0 \le r \le 2s - 2 - t$, and
\[
 r + \sum_{i = 2s - t}^{2s} i = \binom{s}{2}.
\]
Moreover, if $s \ge 2$, then $t \le s-3$.
\end{lem}

\begin{proof}
Observe that $t$ has to be  the largest integer $j$ such that
\[
\sum_{i = 2s - j}^{2s} i \le \binom{s}{2}.
\]
This provides existence and uniqueness of $t$ and $r$.

Furthermore, if $s \ge 3$, then $\sum_{i = s+3}^{2s} i \ge
\binom{s}{2}$. Noting also that $t = -1$ if $s=2$, the final claim
follows.
\end{proof}

\begin{prop}
  \label{prop:stable-ideal-for-symm}
Adopt the assumptions of Theorem \ref{thm:symm}. Set $e := 2 d_m =
\deg a_{m,m}$ and $s = \frac{m-1}{2}$. Let $t, r$ be the integers
such that $t \ge -1,\; 0 \le r \le 2s - 2 - t$, and
\[
 r + \sum_{i = 2s - t}^{2s} i = \binom{s}{2}.
\]
If  $I$ is componentwise linear, then $I$ has the same graded Betti
numbers as the strongly stable ideal $(x_1, x_2, x_3)^{m-1}$,
provided $e = 1$, and as
\begin{eqnarray*}
(x_1, x_2)^{2s-1-t} (x_1, x_2, x_3)^{t+1}
+ x_1^{2s - 1 - t - r} (x_1, x_2)^{r-1} x_3^{t+2}
+ (x_1, x_2)^{2s-2-t-r} x_2^r x_3^{e+t-1} \\
+ (x_1, x_2)^s (x_1, x_2, x_3)^{s-3-t} x_3^{e+t-2}
+ (x_1, x_2, x_3)^{s-1} x_3^{2e + s -1},
\end{eqnarray*}
 where $(x_1, x_2)^{r-1}$ is defined to be zero if
$r=0$, whenever $e \ge 2$.
\end{prop}

\begin{proof}
If $e = 1$, then the claims are clear. Thus, let $e \ge 2$. Then, by
Theorem \ref{thm:symm}, the minimal graded free resolution
\eqref{eq:res-symm} of the ideal $I$ takes the following form
\begin{equation*}
  0 \to F_3 \to F_2 \to F_1 \to I \to 0,
\end{equation*}
where
\begin{eqnarray*}
F_1 = R^{\binom{s+2}{2}} (-m+1) \oplus R^{s (s+1)} (-m+2 -e)
\oplus R^{\binom{s+1}{2}} (-m+3-2e), \\[1ex]
F_2 = R^{s (s+1)}(-m) \oplus R^{2s (s+1)} (-m+1-e)
  \oplus R^{s (s+1)} (-m+2-2e),
\end{eqnarray*}
and
\[
F_3 = R^{\binom{s}{2}} (-m-1) \oplus R^{s (s+1)} (-m -e)
\oplus R^{\binom{s+1}{2}} (-m+1-2e).
\]
Using the Eliahou-Kervaire resolution, one sees that $I$ has the
same graded Betti numbers as the monomial ideal, given in the
statement. Furthermore, one checks that this monomial ideal is
strongly stable.
\end{proof}

%
%

\end{document}